\documentclass[11pt]{article}
\usepackage[T1]{fontenc}
\usepackage[utf8]{inputenc}
\usepackage[english]{babel}
\usepackage{amsmath,amssymb,amsthm}
\usepackage[a4paper,margin=2.5cm]{geometry}
\usepackage[unicode]{hyperref}
\usepackage{enumitem}

\newtheorem{theorem}{Theorem}
\newtheorem{lemma}{Lemma}
\newtheorem{proposition}{Proposition}
\newtheorem{fact}{Fact}

\newtheorem{question}{Question}
\theoremstyle{definition}
\newtheorem{example}{Example}
\theoremstyle{remark}
\newtheorem{remark}{Remark}

\newcommand{\R}{\mathbb R}
\newcommand{\supp}{\operatorname{supp}}

\title{Hemicompactness and the $k_\omega$-property of free\\ topological vector spaces}
\author{E.\,A.~Reznichenko}
\date{}

\begin{document}
\maketitle

\begin{abstract}
For a Tychonoff space $X$ we consider the vector space $E(X)$ with basis
$X$ endowed with one of the following topologies: $L_b(X)$ and $L_p(X)$
(the weak topologies generated by $C_b(X)$ and $C(X)$), $L(X)$ (the free
locally convex space), $L^{(r)}(X)$ (the free locally $r$-convex space,
$0<r\le1$), $L^{(0)}(X)$ (the free locally pseudoconvex space) and $V(X)$
(the free topological vector space). We prove that $L_b(X)$ is
hemicompact if and only if $X$ is finite; that for every vector topology
between $L_p(X)$ and $L^{(0)}(X)$ hemicompactness is equivalent to $X$
being countable with all compact subsets finite; and that $V(X)$, as
well as the free topological groups $F(X)$ and $A(X)$, are hemicompact
if and only if $X$ is hemicompact. In particular, $V([0,1])$ is
hemicompact, while $L([0,1])$ and $L_p([0,1])$ are not. We also
describe when these spaces are $k_\omega$-spaces: $L_b(X)$ and $L_p(X)$
only for finite $X$, and $L^{(r)}(X)$, $0\le r\le1$, only for countable
discrete $X$.
\end{abstract}

\section{Introduction}

A space $Z$ is called \emph{hemicompact} if it has a sequence of compact
sets $K_1\subseteq K_2\subseteq\dots$ such that every compact set
$K\subseteq Z$ is contained in some $K_n$. If, in addition, a set
$F\subseteq Z$ is closed if and only if $F\cap K_n$ is closed in $K_n$
for all $n$, then $Z$ is called a \emph{$k_\omega$-space}. Every
hemicompact space is $\sigma$-compact, and the $k_\omega$-spaces are
exactly the hemicompact $k$-spaces \cite{Engelking}. A Tychonoff space
$X$ is hemicompact if and only if the space $C_k(X)$ of continuous
functions on $X$ with the compact-open topology is metrizable
(Arens~\cite{Arens}).

Let $X$ be a Tychonoff space and $E(X)$ the vector space with basis $X$.
We consider the following topologies on $E(X)$:
\begin{itemize}
\item $L_b(X)$ and $L_p(X)$ are the weak topologies $\sigma(E(X),C_b(X))$
and $\sigma(E(X),C(X))$;
\item $L(X)$ is the topology of the free locally convex space;
\item $L^{(r)}(X)$, $0<r\le1$, is the topology of the free locally
$r$-convex space; for $r=1$ it coincides with $L(X)$;
\item $L^{(0)}(X)$ is the topology of the free locally pseudoconvex
space;
\item $V(X)$ is the topology of the free topological vector space.
\end{itemize}
A TVS is called \emph{locally $r$-convex} if its topology is generated by
$r$-seminorms ($q(x+y)\le q(x)+q(y)$, $q(\lambda x)=|\lambda|^rq(x)$),
and \emph{locally pseudoconvex} if its topology is the supremum of
locally $r$-convex topologies (with, in general, different $r$). Thus
$L^{(r)}(X)$ is the strongest locally $r$-convex topology on $E(X)$ for
which the embedding $X\to E(X)$ is continuous, $L^{(1)}(X)=L(X)$, and
$L^{(0)}(X)=\sup_{r>0}L^{(r)}(X)$. These spaces are described in more
detail in \S\,\ref{sec:Lr}. For $0<s<r\le1$
\[
V(X)\to L^{(0)}(X)\to L^{(s)}(X)\to L^{(r)}(X)\to L(X)\to
L_p(X)\to L_b(X),
\]
where an arrow $Q_1(X)\to Q_2(X)$ means that the identity map of $E(X)$
is continuous as a map from $Q_1(X)$ to $Q_2(X)$, i.e., the topology
$Q_1(X)$ is stronger than $Q_2(X)$. Below, for topologies on $E(X)$ we
also write $\tau_2\le\tau_1$ if $\tau_1$ is stronger than $\tau_2$. For
the topologies $L_b(X)$, $L_p(X)$, $L(X)$, $V(X)$ and their properties
see \cite[\S\,2]{LRS-TVS}; we add to this scale of topologies on $E(X)$
the topologies $L^{(r)}(X)$, $0<r<1$, and $L^{(0)}(X)$.

The following is known about the $k_\omega$-property for spaces of the
form $E(X)$. Gabriyelyan and Morris proved that $V(X)$ is a
$k_\omega$-space if and only if $X$ is a $k_\omega$-space
\cite[Theorem~3.1]{GM2017}; for the free topological groups $F(X)$ and
$A(X)$ the analogous result is due to Mack, Morris and Ordman
\cite{MMO73}. Gabriyelyan \cite{G14} showed that $L(X)$ is a $k$-space if
and only if $X$ is countable and discrete; in this case $L(X)$ coincides
with the $k_\omega$-space $\varphi$ (the inductive limit of $\R^n$), so
$L(X)$ is a $k_\omega$-space exactly for countable discrete $X$.

The spaces $L(X)$ and $V(X)$ are metrizable (and also Baire) only for
finite $X$ \cite{LRS-TVS,GM2017} (see also
Proposition~\ref{prop:nolb}); $L(X)$ is complete if and only if $X$ is
Dieudonn\'e complete and contains no infinite compact sets
(Uspenskii~\cite{Usp83}).

In this paper, for each of the topologies listed above we describe those
$X$ for which the space $E(X)$ is hemicompact
(Theorems~\ref{thm:Lb}--\ref{thm:V}); for the groups $F(X)$ and $A(X)$
the answer is given by Theorem~\ref{thm:groups}. Since the
$k_\omega$-spaces are the hemicompact $k$-spaces, this also yields a
description of the $k_\omega$-spaces among $L_b(X)$, $L_p(X)$ and
$L^{(r)}(X)$, $0\le r\le1$ (Theorem~\ref{thm:komega}).

\begin{theorem}\label{thm:Lb}
$L_b(X)$ is hemicompact if and only if $X$ is finite.
\end{theorem}

\begin{theorem}\label{thm:Lp}
Let $\tau$ be a vector topology on $E(X)$ between $L_p(X)$ and
$L^{(0)}(X)$ \textup{(}for instance, $L_p(X)$, $L(X)$, $L^{(r)}(X)$ with
$0<r<1$, or $L^{(0)}(X)$\textup{)}. The space $(E(X),\tau)$ is
hemicompact if and only if $X$ is countable and all compact subsets of
$X$ are finite.
\end{theorem}

\begin{theorem}\label{thm:V}
$V(X)$ is hemicompact if and only if $X$ is hemicompact.
\end{theorem}

\begin{theorem}\label{thm:groups}
The free topological group $F(X)$ and the free abelian topological group
$A(X)$ are hemicompact if and only if $X$ is hemicompact.
\end{theorem}

For the $k_\omega$-property the answer is as follows.

\begin{theorem}\label{thm:komega}
\begin{enumerate}[label=\textup{(\arabic*)}]
\item $L_b(X)$ and $L_p(X)$ are $k_\omega$-spaces if and only if $X$ is
finite.
\item For $0\le r\le1$ the space $L^{(r)}(X)$ \textup{(}in particular,
$L(X)$ and $L^{(0)}(X)$\textup{)} is a $k_\omega$-space if and only if
$X$ is countable and discrete.
\item \textup{(}Gabriyelyan--Morris \cite[Theorem~3.1]{GM2017}\textup{)}
$V(X)$ is a $k_\omega$-space if and only if $X$ is a $k_\omega$-space.
\end{enumerate}
\end{theorem}

A countable space all of whose compact subsets are finite is
hemicompact; the converse fails (a convergent sequence). Hence the
condition of Theorem~\ref{thm:Lp} is strictly stronger than
hemicompactness of a countable $X$. Thus, for $X=\mathbb N$ all the
spaces under consideration except $L_b(\mathbb N)$ are hemicompact; for
an infinite compact $X$ only $V(X)$ is; for an uncountable discrete $X$
none of them is.

For every topology $\tau$ between $L_b(X)$ and $V(X)$ the space
$(E(X),\tau)$ is $\sigma$-compact if and only if $X$ is
$\sigma$-compact: if $X=\bigcup_nK_n$ with compact $K_n$, then $E(X)$ is
the union of the continuous (already in $V(X)$) images of
$K_n^n\times[-n,n]^n$, and $X$ is closed in $(E(X),\tau)$
(Fact~\ref{f:known}(1)). For the stronger properties, hemicompactness and
$k_\omega$, Theorems~\ref{thm:Lb}--\ref{thm:V} and~\ref{thm:komega} give
the following table:
\begin{center}
\renewcommand{\arraystretch}{1.25}
\begin{tabular}{l|c|c}
 & $k_\omega$ & hemicompact\\ \hline
$L_b(X)$ & $X$ finite & $X$ finite\\
$L_p(X)$ & $X$ finite & $X$ countable, compacta finite\\
$L(X)$, $L^{(r)}(X)$, $L^{(0)}(X)$ & $X$ countable and discrete & $X$ countable, compacta finite\\
$V(X)$ & $X$ is $k_\omega$ \cite{GM2017} & $X$ hemicompact
\end{tabular}
\end{center}

The interval of topologies in Theorem~\ref{thm:Lp} cannot be extended
either downwards (to $L_b(X)$) or upwards (to $V(X)$): outside it,
hemicompactness is no longer determined by the space $X$ alone
(Example~\ref{ex:products}).

For $L_p(X)=C_p(X)'_\sigma$, Theorem~\ref{thm:Lp} agrees with the
characterization of hemicompact weak$^*$ duals from~\cite{R-hemi}: for a
locally convex space $F$, the space $F'_\sigma$ is hemicompact if and
only if the Mackey topology $\tau(F,F')$ is metrizable and barrelled
(Remark~\ref{rem:Cp}). We give a direct proof that also applies to
topologies that are not weak.

\section{Notation and known facts}

For $\chi=\sum_{i\le n}\lambda_ix_i\in E(X)$ with pairwise distinct $x_i$
and $\lambda_i\ne0$ we put $\supp\chi=\{x_1,\dots,x_n\}$ and
$\|\chi\|=\sum|\lambda_i|$; for $A\subseteq E(X)$ let
$\supp A=\bigcup_{\chi\in A}\supp\chi$. Put
$B_n(X)=\{\chi:\|\chi\|\le n\}$,
$\mathrm{sp}_n(X)=\{\sum_{i\le n}\lambda_ix_i:\ \lambda_i\in[-n,n],\
x_i\in X\}$, and for $Y\subseteq X$ denote by $E(Y)\subseteq E(X)$ the
subspace of elements with support in $Y$. The map
$\chi\mapsto(f\mapsto\chi(f))$ identifies $E(X)$ with a subspace of the
algebraic dual $C(X)^*$; all the topologies under consideration except
$L_b(X)$ are stronger than $L_p(X)=\sigma(E(X),C(X))$.

A subset $A\subseteq X$ is \emph{functionally bounded} if every
$f\in C(X)$ is bounded on $A$; $X$ is a \emph{$\mu$-space} if the
closures of functionally bounded sets are compact.

\begin{fact}\label{f:known}
\begin{enumerate}[label=\textup{(\arabic*)}]
\item Each of the topologies under consideration induces the original
topology on $X$, and $X$ is closed in $L_b(X)$, hence in all stronger
topologies. (Indeed, $X$ is closed in $L_p(\beta X)$, and the closure of
$X$ in $L_b(X)\subseteq L_p(\beta X)$ lies in $\beta X\cap E(X)=X$.)
\item A set $A\subseteq L_b(X)$ is bounded if and only if
$A\subseteq B_n(X)$ for some $n$ \cite[Proposition~6.1]{LRS-TVS}.
\item If $A\subseteq L_p(X)$ is bounded, then $\supp A$ is functionally
bounded in $X$ \cite[Lemma~7.5]{LRS-TVS}.
\item Every compact set in $V(X)$ lies in some $\mathrm{sp}_n(X)$
\cite[Corollary~3.4]{GM2017}.
\item A Lindel\"of (in particular, countable or $\sigma$-compact)
Tychonoff space is realcompact and Dieudonn\'e complete, hence a
$\mu$-space \cite{Engelking}.
\end{enumerate}
\end{fact}

\section{The spaces $L^{(r)}(X)$ and $L^{(0)}(X)$}\label{sec:Lr}

A TVS is called \emph{locally $r$-convex} ($0<r\le1$) if zero has a base
of absolutely $r$-convex neighbourhoods $U$ (i.e., $ax+by\in U$ whenever
$x,y\in U$ and $|a|^r+|b|^r\le1$); equivalently, its topology is
generated by $r$-seminorms ($q(x+y)\le q(x)+q(y)$,
$q(\lambda x)=|\lambda|^rq(x)$); see \cite[\S\,3.1]{Rolewicz}, where
these notions are considered for metrizable TVS. The supremum of a
family of locally $r$-convex topologies on $E(X)$ is locally $r$-convex:
it is generated by the union of the corresponding families of
$r$-seminorms. Therefore $L^{(r)}(X)$ exists: it is $E(X)$ with the
supremum of all locally $r$-convex topologies for which the embedding
$X\to E(X)$ is continuous. This topology is stronger than $L(X)$ and
hence induces the original topology on $X$; the universal property
(every continuous map of $X$ into a locally $r$-convex space extends to
a continuous linear map) is verified as for $L(X)$. Since for $s<r$ an
absolutely $r$-convex set is absolutely $s$-convex, the topology
$L^{(r)}(X)$ is weaker than $L^{(s)}(X)$.

A TVS is called \emph{locally pseudoconvex} if zero has a base of
balanced neighbourhoods $U$ with $U+U\subseteq c_UU$
\cite[\S\,3.1]{Rolewicz}. The construction of
\cite[Theorem~3.1.4]{Rolewicz} deals with a single neighbourhood and does
not use metrizability: from a balanced neighbourhood $U$ with
$U+U\subseteq2^{1/r}U$ it produces a continuous $r$-seminorm whose unit
ball lies in $U$. Therefore the locally pseudoconvex topologies are
exactly the suprema of locally $r$-convex topologies (with different
$r$), and the class of such topologies is closed under suprema.
Consequently, the \emph{free locally pseudoconvex space} $L^{(0)}(X)$
exists, and its topology is the supremum of the topologies $L^{(r)}(X)$
over all $r>0$:
\[
L^{(0)}(X)=\sup_{r>0}L^{(r)}(X)=\sup_{m\in\mathbb N}L^{(1/m)}(X).
\]
Since the topologies $L^{(r)}(X)$ increase as $r\downarrow0$, a set
$U\subseteq E(X)$ is a neighbourhood of zero in $L^{(0)}(X)$ if and only
if it is a neighbourhood of zero in some $L^{(r)}(X)$; in particular,
$h_k\to0$ in $L^{(0)}(X)$ if and only if $h_k\to0$ in all $L^{(r)}(X)$.

All the topologies $L^{(r)}(X)$, $0\le r\le1$, lie between $L(X)$ and
$V(X)$; in particular, all of them induce the original topology on $X$,
and $X$ is closed in each of them (Fact~\ref{f:known}(1)). Note also
that a continuous $r$-seminorm $\rho$ on $L^{(r)}(X)$ is bounded on each
compact set $K\subseteq X$ by some $M$, and then
$\rho(\sum c_jy_j)\le M\sum|c_j|^r$ for $y_j\in K$; this inequality is
the main property of $L^{(r)}(X)$ used in the proof of
Theorem~\ref{thm:Lp}.

It is natural to ask also about the free locally bounded space. A
locally bounded space is locally $r$-convex for some $r>0$ (the
Aoki--Rolewicz theorem \cite[Theorem~3.2.1]{Rolewicz}), and each
$L^{(r)}(X)$ is the supremum of the topologies generated by a single
continuous $r$-seminorm, i.e., of locally bounded topologies. Therefore
the supremum of all locally bounded vector topologies on $E(X)$ for which
the embedding $X\to E(X)$ is continuous equals $L^{(0)}(X)$. However,
the free locally bounded space itself does not exist.

\begin{proposition}\label{prop:nolb}
Let $X$ be an infinite Tychonoff space. Then no Hausdorff vector topology
$\tau$ on $E(X)$ that is stronger than $L_p(X)$ and induces the original
topology on $X$ is metrizable. In particular, there is no free locally
bounded space over $X$.
\end{proposition}

\begin{proof}
Let $\tau$ be metrizable. Recall that in a metrizable TVS for every
sequence $(y_n)$ there are $a_n>0$ with $a_ny_n\to0$, and if $y_n\to0$,
then one can choose $a_n\to\infty$. If $X$ is not metrizable, then
neither is $\tau$, since it induces the original topology on $X$. If $X$
is metrizable, then it contains either a nontrivial convergent sequence
$c_n\to c$ (with pairwise distinct $c_n\ne c$) or an infinite closed
discrete set $\{d_n\}$. In the first case $y_n=c_n-c\to0$ in $\tau$;
take $a_n\to\infty$ with $a_ny_n\to0$ and $f\in C(X)$ with $f(c)=0$,
$f(c_n)=a_n^{-1/2}$; then $(a_ny_n)(f)=a_n^{1/2}\not\to0$, which
contradicts the continuity of $f$ on $L_p(X)\le\tau$. In the second case
take $a_n>0$ with $a_nd_n\to0$ and $f\in C(X)$ with $f(d_n)=1/a_n$; then
$(a_nd_n)(f)=1\not\to0$.

Now let $(G,i)$ be a free locally bounded space over $X$. For a finite
$S\subseteq X$ there is a continuous $g\colon X\to\R^S$ mapping $S$ onto
a basis of $\R^S$; hence $i(X)$ is linearly independent and the linear
span of $i(X)$ in $G$ can be identified with $E(X)$. Every $f\in C(X)$
extends to a continuous functional on $G$, so the topology on $E(X)$
induced from $G$ is stronger than $L_p(X)$ and induces the original
topology on $X$. But a subspace of a locally bounded Hausdorff TVS is
locally bounded and hence metrizable, a contradiction.
\end{proof}

\section{Proofs}

\begin{lemma}\label{lem:balls}
Let $Z$ be a Hausdorff TVS that is hemicompact with respect to an
increasing sequence of compact sets $(K_n)$, let $H\subseteq Z$ be a
linear subspace, and let $q\colon H\to[0,\infty)$ be a function such
that $q(h_k)\to0$ implies $h_k\to0$ in $Z$. Then
$\{h\in H:\ q(h)\le\varepsilon\}\subseteq K_N$ for some $\varepsilon>0$
and $N$.
\end{lemma}

\begin{proof}
Otherwise for each $k$ there is $h_k\in H\setminus K_k$ with
$q(h_k)<1/k$. Then $h_k\to0$, and $C=\{0\}\cup\{h_k\}$ is a compact set
in $Z$. If $C\subseteq K_m$, then $h_m\in K_m$, a contradiction.
\end{proof}

The lemma will be applied in the following form: if $Z$ is mapped
continuously and injectively into a Hausdorff space $W$, then $K_N$ is
closed in $W$, and therefore the closure of the ball
$\{q\le\varepsilon\}$ in $W$ lies in the image of $Z$. A similar argument
is used in~\cite{R-hemi} to prove quasi-completeness of hemicompact
weak$^*$ duals.

\begin{proof}[Proof of Theorem~\ref{thm:Lb}]
If $X$ is finite, then $L_b(X)\cong\R^X$. Let $L_b(X)$ be hemicompact.
Put $H=E(X)$ and $q=\|\cdot\|$. Since $|\chi(f)|\le\|f\|_\infty\|\chi\|$
for $f\in C_b(X)$, $\|h_k\|\to0$ implies $h_k\to0$ in $L_b(X)$. By
Lemma~\ref{lem:balls}, $\varepsilon B_1(X)\subseteq K_N$ for some compact
set $K_N\subseteq L_b(X)$. Identify $C_b(X)$ with $C(\beta X)$; then
$L_b(X)\subseteq(C(\beta X)^*,\sigma(C(\beta X)^*,C(\beta X)))$, and
$K_N$ is closed in this space. Finitely supported probability measures
on $X$ are $w^*$-dense in the set $P(\beta X)$ of Radon probability
measures on $\beta X$, hence $\varepsilon P(\beta X)\subseteq K_N\subseteq
E(X)$. In particular, $\delta_p\in E(X)$ for all $p\in\beta X$, i.e.,
$\beta X=X$; and if $X$ were infinite, the measure
$\sum_n2^{-n}\delta_{x_n}$ with distinct $x_n\in X$ would not be finitely
supported. Hence $X$ is finite.
\end{proof}

\begin{proof}[Proof of Theorem~\ref{thm:Lp}]
Let $Z=(E(X),\tau)$.

\emph{Sufficiency.} Let $X$ be countable with all compact subsets finite.
By Fact~\ref{f:known}(5), all functionally bounded subsets of $X$ are
finite. A compact set $C\subseteq Z$ is compact in $L_p(X)$ and hence
bounded there; by Fact~\ref{f:known}(3), $\supp C$ is finite. Thus
$C\subseteq E(S)$ for a finite $S\subseteq X$; on the finite-dimensional
space $E(S)\cong\R^S$ the topology $\tau$ is Euclidean, and $C$ lies in
some cube $[-n,n]^S$. Since there are countably many finite
$S\subseteq X$, $Z$ is hemicompact with respect to the countable family
of cubes $[-n,n]^S$.

\emph{Necessity.} Let $Z$ be hemicompact. By Fact~\ref{f:known}(1), $X$
is closed in $Z$ and hence hemicompact, in particular
$\sigma$-compact. It suffices to show that every compact set
$K\subseteq X$ is finite: then $X$ is a countable union of finite sets.

Suppose that $K$ is infinite. Choose in it an infinite discrete subspace
$\{x_i:i\in\omega\}$ of pairwise distinct points. Put $H=E(K)$ and, for
$h=\sum c_jy_j$,
\[
q_t(h)=\sum_j|c_j|^t\quad(t>0),\qquad
q(h)=\sum_{m\ge1}2^{-m}\min\bigl(1,q_{1/m}(h)\bigr).
\]
If $\rho$ is a continuous $t$-seminorm on $L^{(t)}(X)$, then the function
$x\mapsto\rho(x)$ is continuous on $X$ and bounded on $K$ by some $M$, so
$\rho\le Mq_t$ on $E(K)$. Therefore $q(h_k)\to0$ implies
$q_{1/m}(h_k)\to0$ and $h_k\to0$ in $L^{(1/m)}(X)$ for all $m$, i.e.,
$h_k\to0$ in $L^{(0)}(X)$, and hence in $Z$. By Lemma~\ref{lem:balls},
$\{h\in E(K):q(h)\le\varepsilon\}\subseteq K_N$ for some compact set
$K_N\subseteq Z$. Choose $m_0$ with $2^{-m_0}\le\varepsilon/2$ and
$c\in(0,1)$ with $c^{1/m_0}\sum_ie^{-i^2/m_0}\le\varepsilon/2$; since
$c^{1/m}\le c^{1/m_0}$ and $e^{-i^2/m}\le e^{-i^2/m_0}$ for $m\le m_0$,
for $s_n=c\sum_{i\le n}e^{-i^2}x_i$ we get $q(s_n)\le\varepsilon$, i.e.,
$s_n\in K_N$ for all $n$. The map $Z\to(C(X)^*,\sigma(C(X)^*,C(X)))$ is
continuous and injective, so $K_N$ is closed there as well, and the limit
$\nu=\lim s_n$, $\nu(f)=c\sum_ie^{-i^2}f(x_i)$, lies in
$K_N\subseteq E(X)$. But $\nu\notin E(X)$: if $\nu$ had finite support
$S$, then, taking $j$ with $x_j\notin S$ and $f\in C(X)$ with $f(x_j)=1$
and $f=0$ on $S\cup\overline{\{x_i:i\ne j\}}$ (which is possible since
$x_j$ is isolated in $\{x_i\}$), we would get $\nu(f)=ce^{-j^2}\ne0$,
whereas every element with support in $S$ vanishes at $f$. A
contradiction.
\end{proof}

\begin{proof}[Proof of Theorem~\ref{thm:V}]
If $V(X)$ is hemicompact, then so is $X$, which is closed in it
(Fact~\ref{f:known}(1)). Conversely, let $K_1\subseteq K_2\subseteq\dots$
be cofinal compact sets in $X$. The sets
\[
P_m=\Bigl\{\sum_{i\le m}\lambda_ix_i:\ x_i\in K_m,\ \lambda_i\in[-m,m]\Bigr\}
\]
are compact in $V(X)$ as continuous images of $K_m^m\times[-m,m]^m$. Let
$C\subseteq V(X)$ be compact. By Fact~\ref{f:known}(4),
$C\subseteq\mathrm{sp}_n(X)$ for some $n$. Further, $C$ is bounded in
$L_p(X)$, and by Fact~\ref{f:known}(3), $\supp C$ is functionally
bounded. A hemicompact $X$ is $\sigma$-compact and, by
Fact~\ref{f:known}(5), a $\mu$-space, so $\overline{\supp C}$ is compact
and lies in some $K_m$. Consequently, $C\subseteq P_{\max(m,n)}$.
\end{proof}

\begin{proof}[Proof of Theorem~\ref{thm:groups}]
Necessity follows from the closedness of $X$ in $F(X)$ and in $A(X)$.
Conversely, let $(K_m)$ be cofinal compact sets in $X$. The sets
$F_m(K_m)$ (words of length $\le m$ in the alphabet $K_m\cup K_m^{-1}$)
are compact as continuous images of $(K_m\oplus K_m^{-1}\oplus\{e\})^m$.
A compact set $C\subseteq F(X)$ is functionally bounded and hence lies in
some $F_n(Y,X)$, where $Y\subseteq X$ is functionally bounded (the
Dikranjan--Tkachenko theorem \cite[Theorem~7.5.3]{AT08}); since $X$ is
Dieudonn\'e complete (Fact~\ref{f:known}(5)), the closure of the support
of $C$ is compact \cite[Corollary~7.5.6]{AT08} and lies in some $K_m$.
Hence $C\subseteq F_k(K_k)$ with $k=\max(m,n)$. For $A(X)$ the argument
is the same; one can also note that $A(X)$ is a closed subgroup of
$V(X)$ \cite[Proposition~4.1]{GM2017} and apply Theorem~\ref{thm:V}.
\end{proof}

\begin{proof}[Proof of Theorem~\ref{thm:komega}]
We use the fact that the $k_\omega$-spaces are exactly the hemicompact
$k$-spaces. For $L_b(X)$, statement (1) follows from
Theorem~\ref{thm:Lb}, and (3) is \cite[Theorem~3.1]{GM2017}.

Let $\tau\in\{L_p(X),L^{(r)}(X)\}$, $0\le r\le1$, and let $(E(X),\tau)$
be a $k_\omega$-space. By Theorem~\ref{thm:Lp}, $X$ is countable and all
its compact subsets are finite. Since $X$ is closed in $(E(X),\tau)$
(Fact~\ref{f:known}(1)), $X$ is a $k$-space with finite compact subsets,
i.e., $X$ is discrete. Further, every compact set in $(E(X),\tau)$ lies
in a finite-dimensional $E(S)$, $S\subseteq X$ finite (see the proof of
sufficiency in Theorem~\ref{thm:Lp}), so the $k$-topology $\tau$ is the
strongest topology inducing the Euclidean topology on all
finite-dimensional subspaces, i.e., the topology $\varphi$; on a space of
countable dimension it coincides with the strongest vector topology and
with the strongest locally convex topology
\cite[Proposition~4.1.4]{Jarchow}. For $\tau=L_p(X)$ this is impossible
for infinite $X$: in the weak topology subspaces of finite codimension
are neighbourhoods of zero, whereas in $\varphi$ they are not. Hence in
this case $X$ is finite.

Conversely, for a countable discrete $X$ the topologies
$L(X)\le L^{(r)}(X)\le L^{(0)}(X)\le V(X)$ lie between the strongest
locally convex and the strongest vector topologies on $E(X)$, which
coincide; hence $L^{(r)}(X)=\varphi$ is a $k_\omega$-space.
\end{proof}

\section{Remarks and examples}

\begin{remark}[compact sets in different topologies]\label{rem:compacts}
Let $C\subseteq E(X)$.
\begin{enumerate}[label=\textup{(\alph*)}]
\item If $C$ is compact in $L_b(X)$, then $C\subseteq B_n(X)$ for some
$n$ (Fact~\ref{f:known}(2)).
\item If $C$ is compact in $L_p(X)$ (in particular, in $L^{(r)}(X)$,
$0\le r\le1$, or in $V(X)$), then, in addition, $\supp C$ is
functionally bounded (Fact~\ref{f:known}(3)).
\item If $C$ is compact in $V(X)$, then, in addition,
$C\subseteq\mathrm{sp}_n(X)$ for some $n$ (Fact~\ref{f:known}(4)); in
particular, the lengths of supports of elements of $C$ are bounded.
\end{enumerate}
For compact sets in $L^{(r)}(X)$ statement (c) fails
\cite[Example~7.1]{RL}: let $X$ be an infinite compact space, $(x_n)$
distinct points of $X$, and $g_n=2^{-n}(x_0+\dots+x_n)$. A continuous
$r$-seminorm $\rho$ on $L^{(r)}(X)$ is bounded on $X$ by some $M$, hence
$\rho(g_n)\le(n+1)M2^{-nr}\to0$, and $\{0\}\cup\{g_n\}$ is compact in
$L^{(r)}(X)$, while $|\supp g_n|=n+1$. This difference explains the
discrepancy between Theorems~\ref{thm:Lp} and~\ref{thm:V}: in
$L^{(r)}(X)$ small balls of the quasinorm $\sum|c_j|^r$ shrink to zero
and Lemma~\ref{lem:balls} applies, whereas in $V(X)$ they do not.
\end{remark}

\begin{remark}[the Glicksberg property]\label{rem:Glicksberg}
A set $A\subseteq E(X)$ is bounded in $L_p(X)$ if and only if
$A\subseteq B_n(X)$ for some $n$ and $\supp A$ is functionally bounded:
necessity is Fact~\ref{f:known}(2),(3), and sufficiency follows from the
estimate $|\chi(f)|\le n\sup_{\supp A}|f|$. Since the dual of $L(X)$ is
$C(X)$, by Mackey's theorem the spaces $L(X)$ and $L_p(X)$ have the same
bounded sets. Moreover, Gabriyelyan \cite[Theorem~1.2]{Gab-MAP} proved
that $L(X)$ has the \emph{Glicksberg property}: the compact sets of
$L(X)$ and of its weak topology $L_p(X)$ coincide (see also
\cite{BG22}). Therefore all vector topologies $\tau$ between $L_p(X)$ and
$L(X)$ have the same compact sets, and for such $\tau$
Theorem~\ref{thm:Lp} reduces to the case $\tau=L_p(X)$, i.e., to
Remark~\ref{rem:Cp}. For $L^{(r)}(X)$, $0\le r<1$, this argument does not
work (Example~\ref{ex:noGlicksberg}).
\end{remark}

\begin{example}\label{ex:noGlicksberg}
Let $0\le r<1$ and let $s=\{x_k:k\ge2\}\cup\{x_\infty\}$ be a convergent
sequence, $x_k\to x_\infty$. Put
\[
\chi_n=\frac1n\sum_{k=n+1}^{2n}(x_k-x_\infty).
\]
Since $\chi_n(f)\to0$ for every $f\in C(s)$, the set
$C=\{0\}\cup\{\chi_n\}$ is compact in $L_p(s)$, and hence
(Remark~\ref{rem:Glicksberg}) in $L(s)$. We show that $C$ is not compact
in $L^{(r)}(s)$. Choose $r'\in(0,1)$, equal to $r$ if $r>0$, and
consider the $r'$-Banach space $\ell_{r'}$ with the $r'$-norm
$\|y\|=\sum_i|y_i|^{r'}$. The map $x_k\mapsto a_ke_k$,
$x_\infty\mapsto0$, where $a_k=1/\log k$, is continuous on $s$, since
$\|a_ke_k\|=a_k^{r'}\to0$; its linear extension
$T\colon L^{(r')}(s)\to\ell_{r'}$ is continuous, hence continuous also on
$L^{(r)}(s)$. But
\[
\|T\chi_n\|=n^{-r'}\sum_{k=n+1}^{2n}a_k^{r'}\ge
\frac{n^{1-r'}}{(\log2n)^{r'}}\to\infty,
\]
so $T(C)$ is unbounded and $C$ is not compact in $L^{(r)}(s)$. Thus, for
$0\le r<1$ the space $L^{(r)}(s)$ does not have the Glicksberg property
(its dual is still $C(s)$, and its weak topology is $L_p(s)$), and the
families of compact sets of $L^{(r)}(s)$ and $L(s)$ are different,
although the hemicompactness condition is the same for them
(Theorem~\ref{thm:Lp}). The same argument (with the map
$x\mapsto\sum_ka_k\varphi_k(x)e_k$, where $\varphi_k$ are functions with
pairwise disjoint supports shrinking to $x_\infty$) works for every
metrizable $X$ containing a nontrivial convergent sequence.
\end{example}

\begin{remark}\label{rem:Cp}
For $\tau=L_p(X)$, Theorem~\ref{thm:Lp} can also be derived from the
general characterization of \cite{R-hemi}. Absolutely convex compact
sets in $L_p(X)$ are finite-dimensional (by the Baire theorem, since the
sets of measures supported by at most $n$ points are closed in $L_b(X)$
\cite[Lemma~6.3]{LRS-TVS}), so the Mackey topology on $C(X)$ with respect
to $\langle C(X),E(X)\rangle$ coincides with the pointwise topology, and
hemicompactness of $L_p(X)$ is equivalent to $C_p(X)$ being metrizable
($X$ countable) and barrelled (by Buchwalter--Schmets \cite{BS73}: all
functionally bounded subsets of $X$ are finite). For comparison, by
Okunev \cite{Ok89} (see also \cite{KL22}), $L_p(X)$ is $\sigma$-compact
if and only if $X$ is $\sigma$-compact; thus $L_p([0,1])$ is
$\sigma$-compact but not hemicompact.
\end{remark}

\begin{remark}[completeness]
The condition of Theorem~\ref{thm:Lp} is close to Uspenskii's
completeness criterion \cite{Usp83}. A countable $X$ is Dieudonn\'e
complete, hence for countable $X$
\[
L(X)\ \text{is hemicompact}\iff L(X)\ \text{is complete}\iff
\text{all compact subsets of } X \text{ are finite}.
\]
For uncountable $X$ these properties differ: for an uncountable discrete
$X$ the space $L(X)$ is complete but not hemicompact.
\end{remark}

\begin{remark}[the $k$-property]
For the free locally convex space, a criterion of the $k$-property is
also known: $L(X)$ is a $k$-space if and only if $X$ is countable and
discrete \cite{G14}. An analogue of this result for $L^{(r)}(X)$,
$0\le r\le1$, is obtained in \cite{R-kr}. It is shown there that for an
uncountable discrete $D$ and for an infinite compact space $K$ the
spaces $L^{(r)}(D)$ and $L^{(r)}(K)$ are not $k$-spaces (the second case
uses Theorem~\ref{thm:Lp}), and the general case is reduced to the
question of whether the topology of $L^{(r)}(X)$ induces on the subspace
$E(K)$, where $K\subseteq X$ is compact, the topology $L^{(r)}(K)$.
\end{remark}

\begin{remark}[$\kappa$-Fr\'echet--Urysohn spaces]
A space $Y$ is called \emph{$\kappa$-Fr\'echet--Urysohn} if for every
open $U\subseteq Y$ and every point $y\in\overline U$ there is a sequence
of points of $U$ converging to $y$ (the notion is due to Arhangel'skii,
see \cite{LRS-GR}). A hemicompact homogeneous $\kappa$-Fr\'echet--Urysohn
space is locally compact \cite[Corollary~3.4]{LRS-GR}, and a locally
compact TVS is finite-dimensional. Therefore, for infinite $X$ no
hemicompact space $(E(X),\tau)$, in particular $V(X)$ with hemicompact
$X$, is $\kappa$-Fr\'echet--Urysohn.
\end{remark}

\begin{example}[intermediate topologies]\label{ex:products}
The free functors take disjoint sums to products:
$Q(X_1\oplus X_2)=Q(X_1)\times Q(X_2)$ for $Q\in\{L_b,L_p,L,V\}$ (for
$L_b$, because $C_b(X_1\oplus X_2)=C_b(X_1)\times C_b(X_2)$).
\begin{enumerate}[label=\textup{(\arabic*)}]
\item Let $X=\mathbb N\oplus\mathbb N$ and $\tau=L_b(\mathbb N)\times
L_p(\mathbb N)$. Then $L_b(X)\le\tau\le L_p(X)$, the space $L_p(X)$ is
hemicompact, but $(E(X),\tau)$ is not: it contains a closed copy of the
non-hemicompact space $L_b(\mathbb N)$.
\item Let $X=[0,1]\oplus[0,1]$. The topologies $V(X)=V([0,1])\times
V([0,1])$ and $\tau=L([0,1])\times V([0,1])$ lie between $L(X)$ and
$V(X)$; the first is hemicompact, the second is not, since it contains a
closed copy of $L([0,1])$.
\end{enumerate}
Thus, outside the interval $[L_p(X),L^{(0)}(X)]$ hemicompactness of
$(E(X),\tau)$ is not determined by the space $X$ alone.
\end{example}

\section{Questions}

\begin{question}\label{q:between}
Is there a natural vector topology on $E(X)$, stronger than
$L^{(0)}(X)$ \textup{(}i.e., not locally pseudoconvex\textup{)} and
weaker than $V(X)$, for which hemicompactness is determined by the space
$X$, but the condition differs from those of Theorems~\ref{thm:Lp}
and~\ref{thm:V}?
\end{question}

The notion of a natural topology is not strictly mathematical. We give a
more formal version of Question~\ref{q:between}. A class $\mathfrak B$
of TVS containing $\R$ will be called a \emph{prevariety} if it is closed
under isomorphic images, arbitrary products and subspaces, and a
\emph{variety} if, in addition, it is closed under Hausdorff quotients
(varieties of locally convex spaces were studied in~\cite{DMS72}; for
varieties of TVS see \cite[Definition~1.2]{LM20}). Standard arguments
show that for a Tychonoff $X$ and a prevariety $\mathfrak B$ there exists
a free object $L_{\mathfrak B}(X)\in\mathfrak B$, the space $E(X)$ with a
topology such that every continuous map of $X$ into a space from
$\mathfrak B$ extends uniquely to a continuous linear map of
$L_{\mathfrak B}(X)$; moreover, $X$ is a closed Hamel basis in
$L_{\mathfrak B}(X)$ (cf.\ \cite[Theorem~5.1]{DMS72}). All the topologies
on $E(X)$ considered above, except $L_b(X)$, are of this kind: $L_p(X)$,
$L(X)$, $L^{(r)}(X)$, $L^{(0)}(X)$ and $V(X)$ are the free spaces of the
varieties of weak, locally convex, locally $r$-convex, locally
pseudoconvex and all TVS, respectively. Theorem~\ref{thm:Lp} implies
that if $\mathfrak B$ is a prevariety of locally convex spaces, then
$L_{\mathfrak B}(X)$ is hemicompact if and only if $X$ is countable and
all compact subsets of $X$ are finite (indeed,
$L_p(X)\le L_{\mathfrak B}(X)\le L(X)$). For instance, this holds for
the varieties of~\cite{DMS72}, in particular for the varieties of
nuclear and of Schwartz spaces.

\begin{question}\label{q:variety}
Is there a \textup{(}pre\textup{)}variety of TVS $\mathfrak B$
containing all locally pseudoconvex spaces for which the
hemicompactness condition for $L_{\mathfrak B}(X)$ differs from those of
Theorems~\ref{thm:Lp} and~\ref{thm:V}?
\end{question}

Note that $L_{\mathfrak B}(X)=V(X)$ if and only if
$V(X)\in\mathfrak B$; since every TVS $Y$ is a quotient of $V(Y)$, for a
variety $\mathfrak B$ the equality $L_{\mathfrak B}(X)=V(X)$ for all $X$
is equivalent to $\mathfrak B$ being the variety of all TVS. A natural
candidate is the variety $\mathfrak B_0$ generated by the locally
pseudoconvex spaces and the space $L_0[0,1]$ of measurable functions
with convergence in measure. For a convergent sequence
$X=s=\{x_k\}\cup\{x_\infty\}$ the topology $L_{\mathfrak B_0}(s)$ is
strictly stronger than $L^{(0)}(s)$. Indeed, let $A_k\subseteq[0,1]$ be
independent events with $P(A_k)=1/k$ and $g_k=4^k\mathbf 1_{A_k}$. Then
$g_k\to0$ in measure, so the map $x_k\mapsto g_k$, $x_\infty\mapsto0$,
extends to a continuous linear map
$T\colon L_{\mathfrak B_0}(s)\to L_0[0,1]$. The partial sums
$h_n=\sum_{k\le n}2^{-k}(x_k-x_\infty)$ are Cauchy in $L^{(0)}(s)$ (as in
the proof of Theorem~\ref{thm:Lp}), but $T(h_m-h_n)\ge2^n$ on
$\bigcup_{n<k\le m}A_k$, and $P\bigl(\bigcup_{n<k\le m}A_k\bigr)=1-n/m$,
so $T(h_n)$ is not Cauchy in measure. Hence $T$ is not continuous on
$L^{(0)}(s)$; in particular, the ball lemma in the form used in the proof
of Theorem~\ref{thm:Lp} does not apply to $L_{\mathfrak B_0}$. We do not
know whether $\mathfrak B_0$ is the variety of all TVS and for which $X$
the space $L_{\mathfrak B_0}(X)$ is hemicompact.

\begin{question}
Let $0\le r<1$. What do the compact sets in $L^{(r)}(X)$ look like? In
particular, is it true that for $0<r<1$ every compact set in
$L^{(r)}(X)$ lies in the set $\{\sum_jc_jy_j:\ \sum_j|c_j|^r\le n\}$ for
some $n$?
\end{question}

\end{document}